\documentclass[10pt]{amsart}
\usepackage[margin=0.8in]{geometry}
\usepackage{amssymb,mathrsfs,graphicx,extpfeil}
\usepackage{epsfig}
\usepackage{indentfirst, latexsym, amssymb, enumerate,amsmath,graphicx}
\usepackage{float}
\usepackage{colortbl}
\usepackage{epsfig,subfigure}
\usepackage{dsfont}

\usepackage{amsmath}
\title[Asymptotic analysis for a kinetic-fluid system]{Asymptotic analysis for Vlasov-Fokker-Planck/compressible Navier-Stokes equations with a density-dependent viscosity}

\author[Choi]{Young-Pil Choi}
\address[Young-Pil Choi]{\newline Department of Mathematics and Institute of Applied Mathematics \newline Inha University, Incheon 22212, Korea (Republic of)}
\email{ypchoi@inha.ac.kr}

\author[Jung]{Jinwook Jung}
\address[Jinwook Jung]{\newline Department of Mathematical Sciences \newline Seoul National University, Seoul  08826, Korea (Republic of)}
\email{warp100@snu.ac.kr}

\newtheorem{theorem}{Theorem}[section]
\newtheorem{lemma}{Lemma}[section]

\newtheorem{remark}{Remark}[section]

\newtheorem{definition}{Definition}[section]

\newcommand{\bbr}{\mathbb R}

\newcommand{\D}{\mathbb D}
\newcommand{\R}{\mathbb R}
\newcommand{\e}{\varepsilon}
\newcommand{\lt}{\left}
\newcommand{\rt}{\right}
\newcommand{\mc}{\mathcal{C}}

\newcommand{\bq}{\begin{equation}}
\newcommand{\eq}{\end{equation}}

\def\charf {\mbox{{\text 1}\kern-.30em {\text l}}}

\begin{document}

\date{\today}

%
%
%
\begin{abstract}
We study a hydrodynamic limit of a system of coupled kinetic and fluid equations under a strong local alignment force and  a strong Brownian motion. More precisely, we consider the Vlasov-Fokker-Planck type equation and compressible Navier-Stokes equations with a density-dependent viscosity. Based on a relative entropy argument, by assuming the existence of weak solutions to that kinetic-fluid system, we rigorously derive a two-phase fluid model consisting of isothermal Euler equations and compressible Navier-Stokes equations with a density-dependent viscosity. 
\end{abstract}
\maketitle \centerline{\date}

\allowdisplaybreaks


\section{Introduction}
\setcounter{equation}{0}
The present work is devoted to the asymptotic analysis of a system of kinetic-fluid equations, namely Vlasov-Fokker-Planck equation with a local alignment force coupled with compressible Navier-Stokes equations through the drag force. This system describes the time evolution of dispersed particles immersed in a compressible fluid, in which particles interact with each other via local alignment forces, and particles and fluid are interacting through the drag force. To be more precise,  let $f = f(x,\xi,t)$ be the number density function on the phase point $(x,\xi) \in \R^d \times \R^d$ at time $t \in \R_+$, and $n = n(x,t)$ and $v = v(x,t)$ be the local mass density and the bulk velocity of the compressible fluid, respectively. Then, our main system is governed by 
\begin{align}
\begin{aligned}\label{A-1}
&\partial_t f + \xi \cdot \nabla_x f + \nabla_\xi \cdot ((v-\xi)f) = \nabla_\xi \cdot (\nabla_\xi f - (u-\xi)f), \quad (x,\xi) \in \R^d \times \R^d, \quad t >0,\\
&\partial_t n + \nabla_x \cdot (nv) = 0,\\
&\partial_t (nv) + \nabla_x \cdot (nv \otimes v) + \nabla_x p - 2\nabla_x \cdot (\nu(n) \D v) = -\int_{\R^d} (v - \xi)f\,d\xi, \\
\end{aligned}
\end{align}
subject to initial data:
\[
f(x,\xi,0) = f_0(x,\xi), \quad n(x,0) = n_0(x), \quad v(x,0) = v_0(x), \qquad (x,\xi) \in \R^d \times \R^d,
\]
and the boundary conditions:
\[
f(x,\xi,t) \to 0, \quad n(x,t) \to n_\infty \in \R_+, \quad \mbox{and} \quad v(x,t) \to 0,
\]
sufficiently fast as $|x|$, $|\xi| \to \infty$, where $\D v$ is the deformation tensor given by $\D v := (\nabla_x v + (\nabla_x v)^t) / 2$, $\nu$ is the viscosity coefficient which is a function of the fluid density $n$, $p=p(n) := n^\gamma$ $(\gamma > 0)$ is the pressure law, and $\rho$ and $u$ denote the average local densitiy and velocity of $f$, respectively:
\[ 
\rho(x,t) := \int_{\bbr^d} f(x,\xi,t)\, d\xi \quad \mbox{and} \quad (\rho u)(x,t) := \int_{\bbr^d} \xi f(x,\xi,t)\, d\xi. 
\]
Those types of kinetic-fluid systems have been extensively studied. The global well-posedness of weak and strong solutions for the Vlasov-type kinetic equations coupled with the incompressible Navier-Stokes equations are discussed in \cite{B-C-H-K2, B-D-G-M, C-C-K, C-D-M, CHJKpre,  G-H-M-Z, WY14, Yu13} and coupled with compressible Navier-Stokes \cite{B-C-H-K1, C-G, CHJKpre2, M-V, M-V2}. The local-in-time existence of classical solutions for the Vlasov-Boltzmann/compressible Euler equations is obtained in \cite{Math10}, and more recently, the global-in-time existence of weak solutions for the BGK/incompressible Navier-Stokes is also discussed in \cite{CYpre}. We refer to \cite{C3} and \cite{C4} for a priori estimate of large-time behavior of solutions and the finite-time blow-up phenomena of Vlasov-type/Navier-Stokes equations, respectively. 

In the current work, we are interested in the asymptotic regime corresponding to a strong drag force and a strong Brownian motion. More specifically, we consider the following system:
\begin{align}
\begin{aligned}\label{B-1}
&\partial_t f^\varepsilon + \xi \cdot \nabla_x f^\varepsilon + \nabla_\xi \cdot ((v^\varepsilon -\xi)f^\varepsilon) = \frac{1}{\varepsilon} \nabla_\xi \cdot (\nabla_\xi f^\varepsilon - (u^\varepsilon - \xi)f^\varepsilon),\\
&\partial_t n^\varepsilon + \nabla_x \cdot (n^\e  v^\varepsilon)=0,\\
&\partial_t (n^\e  v^\varepsilon) + \nabla_x \cdot (n^\e v^\varepsilon \otimes v^\varepsilon) + \nabla_x p(n^\e)  - 2\nabla_x \cdot (\nu(n^\e) \D v^\e) = - \int_{\bbr^d} (v^\varepsilon -\xi)f^\varepsilon d\xi,
\end{aligned}
\end{align}
where
\[ 
\rho^\e(x,t) := \int_{\bbr^d} f^\e(x,\xi,t)\, d\xi \quad \mbox{and} \quad (\rho^\e u^\e)(x,t) := \int_{\bbr^d} \xi f^\e(x,\xi,t)\, d\xi. 
\]
Here, since we are concerned with unbounded domain, we assumed the far-field behavior $n^\e \to n_\infty$ as $|x| \to \infty$ for all $\e \geq 0$. Note that the global-in-time strong solutions to the kinetic equation in \eqref{B-1} around the global Maxwellian is studied in \cite{C2} and the global-in-time existence of weak solutions to the Vlasov-Fokker-Planck/compressible Navier-Stokes equations with a constant viscosity coefficient in a bounded domain is established in \cite{M-V}. Our main purpose is to investigate the convergence of weak solutions $(f^\e, n^\e, v^\e)$ of the above system \eqref{B-1} to the strong solutions $(\rho, u, n, v)$ to the following system of fluid equations:
\begin{align}
\begin{aligned}\label{B-2}
&\partial_t \rho + \nabla_x \cdot (\rho u) = 0,\\
&\partial_t (\rho u) + \nabla_x \cdot (\rho u \otimes u) + \nabla_x \rho = \rho(v-u),\\
&\partial_t n + \nabla_x \cdot (nv) = 0,\\
&\partial_t (nv) + \nabla_x \cdot (nv \otimes v) + \nabla_x p(n) - 2\nabla_x \cdot (\nu(n) \D v)  = -\rho(v-u).
\end{aligned}
\end{align}
The hydrodynamic limit of kinetic equation appeared in \eqref{B-1} coupled with the incompressible Navier-Stokes equations is addressed in \cite{C-C-K} based on the relative entropy method which relies on the ``weak-strong'' uniqueness principle \cite{Da}. The hydrodynamic limit found in \cite{C-C-K} holds as long as there exists a unique strong solution to the limiting system, which is a system of Euler/incompressible Navier-Stokes equations. Later, in \cite{C0}, the global-in-time existence and uniqueness of strong solutions to that limiting system is obtained. We refer to \cite{G-J-V, G-J-V2, M-V2} for other kind of hydrodynamical limits.

Our main strategy relies on the relative entropy argument, which is widely used to analyze a hydrodynamic limit of the kinetic equation \cite{G-L-T, L, Y}, together with some entropy inequalities. In order to establish the hydrodynamic limit, we first need to show the existence of weak and strong solutions to the systems \eqref{B-1} and \eqref{B-2} at least locally in time, and estimate the error between them by means of relative entropy method. However, in the current work, we focus on the relative entropy estimates by assuming the existence of weak solutions to the kinetic-fluid system \eqref{B-1}. Since local existence theories for the types of balance laws have been well developed, the local-in-time existence and uniqueness of solutions for the limiting system \eqref{B-2} can be obtained under suitable assumption on the viscosity coefficient $\nu$, see \cite{M84} for the readers who are interested in it.  We also refer to \cite{C} where the global-in-time existence of a unique strong solution under suitable smallness and regularity assumptions on the initial data is discussed. This yields that once we obtain the existence of weak solutions to the system \eqref{B-1}, our analysis becomes fully rigorous. We emphasize that the asymptotic regime we considered for the system \eqref{B-1} has not been studied so far, to the best of authors' knowledge. 

%
%
%
\subsection{Formal derivation of the asymptotic system}
The right-hand side of the kinetic equation in \eqref{B-1} reads
\[ 
\nabla_\xi \cdot [\nabla_\xi f^\varepsilon - (u^\varepsilon - \xi)f^\varepsilon] = \nabla_\xi \cdot \left(M_{f^\varepsilon} \nabla_\xi \left( \frac{f^\varepsilon}{M_{f^\varepsilon}} \right)\right),
\]
where $M_{f^\varepsilon} = M_{f^\varepsilon}(x,\xi,t)$ is the Maxwellian given by
\[
M_{f^\varepsilon}(x,\xi,t) := \frac{1}{(2\pi)^{d/2}} e^{-\frac{|\xi - u^\varepsilon(x,t)|^2}{2}}. 
\]
Thus, once we have $\rho^\e \to \rho$ and $u^\varepsilon \to u$ as $\e \to 0$, we find
\[
f^\varepsilon \to M_{\rho,u}:= \frac{\rho(x,t)}{(2\pi)^{d/2}} e^{-\frac{|\xi-u(x,t)|^2}{2}} \quad \mbox{as} \quad \varepsilon \to 0. 
\]
This enables us to close the momentum equations derived from the kinetic equation \eqref{B-1} and the limiting solutions $(\rho, u, n, v)$, where $(n^\e, v^\e) \to (n,v)$ as $\e \to 0$, satisfy the two-phase fluid system presented in \eqref{B-2}. See \cite{C-C-K, C} for more detailed discussion. 

Without loss of generality, throughout this paper, we may assume $\|f_0^\varepsilon\|_{L^1} = 1$ for all $\e > 0$. This together with the conservation of mass yields $\|f^\varepsilon(\cdot,\cdot,t)\|_{L^1} = \|f_0^\varepsilon\|_{L^1}$ for $\e>0$ and $t\geq 0$. In fact, we only need to assume $\|f_0^\varepsilon\|_{L^1} \leq C$ for all $\e > 0$, where $C > 0$ is independent of $\e$. 

%
%
%
\subsection{Main result}
For the hydrodynamic limit, we will use the following notion of weak solutions to the system \eqref{A-1} and strong solutions to the system \eqref{B-2}. 
\begin{definition}\label{D2.1}
For $T \in (0,\infty)$, we say a triplet $(f, n, v)$ is a weak solution to the system \eqref{A-1} if the following conditions are satisfied:
\begin{enumerate}
\item
$f \in L^\infty(0,T;(L_+^1 \cap L^\infty)(\bbr^d \times \bbr^d)), \quad (|x|^2 + |\xi|^2) f \in L^\infty(0,T;L^1(\bbr^d \times \bbr^d))$.
\item
$n-n_\infty \in L^\infty(0,T;(L_+^1 \cap L^\gamma)(\bbr^d)), \quad n|v|^2 \in L^\infty(0,T;L^1(\R^d)), \quad  \sqrt{\nu(n)}\nabla_x v \in L^2(0,T;L^2(\bbr^d))$.
\item
$(f,n,v)$ satisfies \eqref{A-1} in a distributional sense.
%
\end{enumerate}
\end{definition}

\begin{definition}\label{D2.2} Let $s >d/2+2$. For $T \in (0,\infty)$, $(\rho,u,n,v)$ is called a strong solution of \eqref{B-2} on the time interval $[0,T]$ if it satisfies the system \eqref{B-2} in the sense of distributions, and it also satisfies the following regularity conditions:
\[
(\rho, u, n, v) \in \mathcal{C}([0,T];H^s(\R^d)) \times \mathcal{C}([0,T]; H^s(\R^d)) \times \mathcal{C}([0,T];H^s(\R^d)) \times \mathcal{C}([0,T];H^s(\R^d)).
\]
\end{definition}
\begin{remark} As discussed before, the local-in-time existence and uniqueness of strong solutions in the sense of Definition \ref{D2.2} can be obtained under suitable assumptions on the initial data and the viscosity coefficient $\nu$.
\end{remark}

We now state our main result on the hydrodynamic limit of \eqref{B-1}. 
\begin{theorem}\label{T2.1}
Let $d > 2$, $\gamma \in [1,2]$, and $(f^\e, n^\e, v^\e)$ be a weak solution to the system \eqref{B-1} up to time $T>0$ in the sense of Definition \ref{D2.1} with the initial data $(f^\e_0, n^\e_0, v^\e_0)$ satisfying 
\begin{align}\label{D2-1.1}
\begin{aligned}
&f^\e_0\in (L_+^1 \cap L^\infty)(\bbr^d \times \bbr^d),  \quad  (|x|^2 + |\xi|^2)f^\e_0 \in L^1(\bbr^d \times \bbr^d),\cr
& n^\e_0-n_\infty\in (L_+^1 \cap L^\gamma)(\bbr^d), \quad n^\e_0|v^\e_0|^2 \in L^1(\R^d), \quad \mbox{and} \quad  \sqrt{\nu(n^\e_0)}\nabla_x v^\e_0 \in L^2(\bbr^d).
\end{aligned}
\end{align}
Let $s > d/2 + 2$ and $(\rho,u,n,v)$ be a strong solution to the system \eqref{B-2} up to time $T>0$ in the sense of Definition \ref{D2.2} with the initial data $(\rho_0,u_0,n_0,v_0)$ satisfying
\[
\rho_0 > 0 \mbox{ in } \R^d, \quad  \inf_{x \in \R^d} n_0(x) > 0, \quad \mbox{and} \quad (\rho_0, u_0, n_0, v_0) \in H^s(\R^d) \times H^s(\R^d) \times H^s(\R^d) \times H^s(\R^d).
\]
Suppose that the viscosity coefficient $\nu \in \mc^1(\R_+)$ is Lipschitz continuous satisfying 
\bq\label{ass_nu}
|\nu(x) - \nu(y)| \leq \nu_{\mbox{\tiny Lip}}|x-y|, \quad \nu(x) \geq \nu_* > 0, \quad \mbox{and} \quad x^2 \leq c_0 \nu(x)p(x), 
\eq
for all $x,y \in \R_+$, where $\nu_{\mbox{\tiny Lip}}$, $\nu_*$, and $c_0$ are positive constants. Moreover, the initial data $(f^\e_0, n^\e_0, v^\e_0)$ and $(\rho_0,u_0,n_0,v_0)$ are well-prepared such that 
\begin{itemize}
\item[{\bf (H1):}] 
$$\begin{aligned}
&\quad \int_{\bbr^d}\lt( \int_{\R^d} f_0^\varepsilon \left(1 +  \log f_0^\varepsilon  +\frac{1}{2}(|\xi|^2+|x|^2) \right) d\xi + \frac{1}{2} n_0^\varepsilon |v_0^\varepsilon|^2 + n_0^\e\int_{n_\infty}^{n_0^\e} \frac{p(z)}{z^2}\,dz - \frac{p(n_\infty)}{n_\infty}(n_0^\e - n_\infty) \rt) dx \cr
&\qquad - \int_{\R^d} \lt(\rho_0 \lt(1 + \log \rho_0 + \frac{1}{2}(|u_0|^2 + |x|^2)\rt) + \frac12 n_0|v_0|^2 + n_0\int_{n_\infty}^{n_0} \frac{p(z)}{z^2}\,dz - \frac{p(n_\infty)}{n_\infty}(n_0 - n_\infty)\rt)dx \cr
&\quad = \mathcal{O}(\sqrt\e),
\end{aligned}$$
\item[{\bf (H2):}] 
$$\begin{aligned}
&\int_{\R^d} \rho_0^\e |u_0^\e - u_0|^2\,dx + \int_{\R^d} n_0^\e|v_0^\e - v_0|^2\,dx \cr
& \qquad + \int_{\R^d}\int_{\rho_0}^{\rho_0^\e} \frac{\rho_0^\varepsilon - z}{z}\,dzdx + \int_{\R^d} \lt(n_0^\e\int_{n_0}^{n_0^\e} \frac{p(z)}{z^2}\,dz - \frac{p(n_0)}{n_0}(n_0^\e - n_0)\rt)dx \cr
&\quad = \mathcal{O}(\sqrt\e).
\end{aligned}$$
\end{itemize}
Then we have 
\begin{align}\label{res_conv}
\begin{aligned}
&\int_{\R^d} \rho^\e|u^\e - u|^2\,dx + \int_{\R^d} (n^\e)|v^\e - v|^2\,dx + \int_{\R^d}\int_{\rho}^{\rho^\e} \frac{\rho^\varepsilon - z}{z}\,dzdx \cr
&\quad + \int_{\R^d} \lt(n^\e\int_n^{n^\e} \frac{p(z)}{z^2}\,dz - \frac{p(n)}{n}(n^\e - n)\rt)dx\cr
&\quad + \int_0^t \int_{\R^d} \nu(n^\e)|\D (v-v^\e)|^2\,dxds +\int_0^t \int_{\bbr^d} \rho^\varepsilon|( u^\varepsilon - v^\varepsilon) - (u-v)|^2\, dxds\\
& \qquad \le C\sqrt{\varepsilon},
\end{aligned}
\end{align}
where $C$ is a positive constant independent of $\e$.

As a consequence, we have the following strong convergences of weak solutions $(f^\varepsilon, n^\varepsilon, v^\varepsilon)$ to the system \eqref{B-1} towards the strong solutions $(\rho,u,n,v)$ to the system \eqref{B-2}:
$$\begin{aligned}
f^\varepsilon &\to M_{\rho,u} \mbox{ a.e. and strongly in }L_{loc}^1(0,T; L^1(\bbr^d \times \bbr^d)),\cr
(\rho^\e, n^\e) &\to (\rho, n) \mbox{ a.e. and strongly in }L_{loc}^1(0,T; L^1(\bbr^d)) \times L_{loc}^1(0,T; L_{loc}^p(\bbr^d))\quad \mbox{for any  } p \in [1,\gamma],\cr
(\rho^\e u^\e, n^\e v^\e) &\to (\rho u, nv) \mbox{ a.e. and strongly in }L_{loc}^1(0,T; L^1(\bbr^d)) \times L_{loc}^1(0,T; L_{loc}^1(\bbr^d)), \quad \mbox{and} \quad \cr
(\rho^\e |u^\e|^2, n^\e |v^\e|^2) &\to (\rho |u|^2, n|v|^2) \mbox{ a.e. and strongly in }L_{loc}^1(0,T; L^1(\bbr^d)) \times L_{loc}^1(0,T; L_{loc}^1(\bbr^d)),
\end{aligned}$$
as $\e \to 0$.
\end{theorem}

\begin{remark} Since $n^\e$ is not integrable in $\R^d$, we only provide the convergences related to the compressible Navier-Stokes system in \eqref{B-1} locally in $\R^d$. 
\end{remark}

\begin{remark} The technical condition $\gamma \in [1,2]$ is also used in \cite{M-V2}, where the asymptotic analysis of the Vlasov-Fokker-Planck equations coupled with the compressible Navier-Stokes equaiton with the constant viscosity coefficient in a bounded domain under strong drag force and strong Brownian motion is studied. 
\end{remark}

\begin{remark} By Young's inequality, we find
\[
r^{2 - \gamma} \leq (\gamma-1) + (2-\gamma)r \quad \mbox{for} \quad \gamma \in [1,2].
\]
This yields that $\nu(r) = 1+r$ satisfies the assumption \eqref{ass_nu} with $\nu_{\mbox{\tiny Lip}} = \nu_* = 1$ and $c_0 = \max(\gamma-1,2-\gamma) > 0$. It looks that the assumptions on $\nu$ \eqref{ass_nu} do not allow us to consider the constant viscosity coefficient. However, if $\nu \equiv \nu_*$ for an example, the third assumption in \eqref{ass_nu} is not needed in our estimate. To be more specific, the term $K_7$ in Section \ref{sec:3} vanishes. Thus our strategy can be directly applied to the constant viscosity coefficient case. 
\end{remark}

\begin{remark} Recently, a non-trivial relative entropy for compressible Navier-Stokes equations with density-dependent viscosities is introduced and some applications, for examples, weak-strong uniqueness, inviscid limit or low Mach number limit, are discussed in \cite{BNV16, BNV17, FJN12}. 
\end{remark}

%

The next section is devoted to derive an evolution equation for the integrated relative entropy. Finally, in Section \ref{sec:3}, we provide the details of proof of Theorem \ref{T2.1}.

Before closing this section, we introduce several notations used throughout the paper. For a function $f = f(x,\xi)$ defined on $(x,\xi) \in \bbr^d \times \bbr^d$, $u = u(x,t)$ on $x \in \bbr^d$ and $p \in [1,\infty)$, we denote $\|f\|_{L^p}$ and $\|u\|_{L^p}$ by the usual $L^p(\bbr^d \times \bbr^d)$- and $L^p(\bbr^d)$-norm, respectively. $H^k(\R^d)$ is the $k$-th order $L^2$-Sobolev space. We also denote by $C$ a generic positive constant which may differ from line to line; $C = C(\alpha, \beta, \cdots)$ represents the positive constants depending on $\alpha, \beta, \cdots$.

%
%
%
\section{Relative entropy estimate}\label{sec:2} \setcounter{equation}{0}
In this section, we present some entropy inequalities and relative entropy estimates which will crucially be used for the hydrodynamic limit of the system \eqref{B-1}.

\subsection{Entropy inequalities}
In this part, we show that the weak solutions to the system \eqref{B-1} in the sense of Definition \ref{D2.1} satisfy several entropy inequalities which are uniform in $\varepsilon$. Similarly to \cite[Section 5]{C-C-K}, let us set
\begin{align*}
&\mathcal{F}(f^\varepsilon, n^\varepsilon, v^\varepsilon) := \int_{\bbr^d \times \bbr^d } f^\varepsilon \left( \log f^\varepsilon  +\frac{|\xi|^2}{2} \right) dxd\xi + \int_{\bbr^d} \frac{1}{2} n^\varepsilon |v^\varepsilon|^2 \,dx + \int_{\bbr^d} H(n^\varepsilon) \,dx,\\
&D_1(f^\varepsilon) := \int_{\bbr^d \times \bbr^d} \frac{1}{f^\varepsilon} |\nabla_\xi f^\varepsilon - (u^\varepsilon -\xi)f^\varepsilon|^2 \,dxd\xi,\\
&D_2(f^\varepsilon, n^\e, v^\varepsilon) := \int_{\bbr^d \times \bbr^d} |v^\varepsilon -\xi|^2 f^\varepsilon \,dxd\xi + \int_{\bbr^d} \nu(n^\e) |\D v^\varepsilon|^2 \,dx,
\end{align*}
where $H = H(n)$ is given by
\[
H(n) := K(n) - K'(n_\infty)(n-n_\infty),  \quad K(n) := n \int_{n_\infty}^n \frac{p(z)}{z^2}\,dz.
\]
Then we can easily find 
\begin{equation}\label{A-2}
\mathcal{F}(f^\varepsilon,n^\varepsilon,v^\varepsilon) + \frac{1}{\varepsilon}\int_0^t D_1(f^\varepsilon)\,ds +\int_0^t D_2(f^\varepsilon, n^\e, v^\varepsilon)\,ds \leq  \mathcal{F}(f^\varepsilon_0,n^\varepsilon_0,v^\varepsilon_0)+  dt \quad \mbox{for} \quad t \geq 0.
\end{equation}
Note that the term $\int_{\R^d \times \R^d} f^\e \log f^\e\,dxd\xi$ has an indefinite sign, however, in the lemma below, we show that it can be controlled by $\mathcal{F}(f^\varepsilon_0,n^\varepsilon_0,v^\varepsilon_0)$ and second spatial moment of $f^\e_0$.
\begin{lemma}\label{L3.1}Let $T>0$ and suppose that $(f^\varepsilon, n^\varepsilon, v^\varepsilon)$ is a weak solution to the system \eqref{B-1} on the time interval $[0,T)$ in the sense of Definition \ref{D2.1} with the initial data $(f^\e_0, n^\e_0, v^\e_0)$ satisfying \eqref{D2-1.1}. Then we have
\begin{align*}
\int_{\R^d \times \R^d} &f^\varepsilon \left( 1 + |\log f^\varepsilon| + \frac{1}{4}(|x|^2 + |\xi|^2) \right) dx d\xi + \frac{1}{2} \int_{\bbr^d} n^\varepsilon |v^\varepsilon|^2 \,dx +  \int_{\bbr^d} H(n^\varepsilon) \,dx\\
& +\frac{1}{\varepsilon} \int_0^t D_1(f^\varepsilon) \,ds + \int_0^t D_2(f^\varepsilon, n^\e, v^\varepsilon) \,ds \le C(T) + \mathcal{O}(\sqrt{\varepsilon}), \qquad t \in (0,T),
\end{align*}
where $C=C(T)$ is a positive constant independent of $\e$. 
\end{lemma}
\begin{proof}It follows from \eqref{A-2} that
\begin{align*}
\frac{d}{dt}&\left(\mathcal{F}(f^\varepsilon,n^\varepsilon,v^\varepsilon) + \int_{\R^d \times \R^d} f^\varepsilon \frac{|x|^2}{2}\, dx d\xi \right) + \frac{1}{\varepsilon} D_1(f^\varepsilon) + D_2(f^\varepsilon, n^\e, v^\varepsilon)\\
&\le \int_{\R^d \times \R^d} f^\varepsilon (x \cdot \xi) \, dx d\xi  + d  \\
&\le \int_{\R^d \times \R^d}\lt( f^\varepsilon \left( \frac{|\xi|^2}{2} + \frac{|x|^2}{2} \right) + 2f^\varepsilon\log_{-}f^\varepsilon - 2f^\varepsilon\log_{-}f^\varepsilon \rt) dx d\xi  + d,
\end{align*}
where $\log_{-}g(x) := \max\{0, -\log g(x)\}$. On the other hand, we get
\[
2\int_{\R^d \times \R^d} f^\varepsilon \log_{-}f^\varepsilon \,dx d\xi  \le \int_{\R^d \times \R^d} f^\varepsilon \left( \frac{|x|^2}{2} + \frac{|\xi|^2}{2} \right) dx d\xi  + \frac{1}{e}\int_{\R^d \times \R^d} e^{-\frac{|\xi|^2}{4} - \frac{|x|^2}{4}} dx d\xi,
\]
and this implies
\[
\frac{d}{dt} \left(\mathcal{F}(f^\varepsilon,n^\varepsilon,v^\varepsilon) + \int_{\R^d \times \R^d} f^\varepsilon \frac{|x|^2}{2} \,dx d \xi \right)\le 2\left( \mathcal{F}(f^\varepsilon,n^\varepsilon,v^\varepsilon) + \int_{\R^d \times \R^d} f^\varepsilon \frac{|x|^2}{2}\, dx d \xi \right) + C,
\]
Thus we obtain
\[
\left(\mathcal{F}(f^\varepsilon,n^\varepsilon,v^\varepsilon)+ \int_{\R^d \times \R^d} f^\varepsilon \frac{|x|^2}{2} \,dx d \xi  \right) \le \left( \mathcal{F}(f^\varepsilon_0,n^\varepsilon_0,v^\varepsilon_0) + \int_{\R^d \times \R^d} f_0^\varepsilon \frac{|x|^2}{2} \,dx d \xi  \right)e^{C(T)}. 
\]
Finally, we combine the above inequality with \eqref{A-2} and {\bf(H1)} to conclude the desired result.
\end{proof}
We now present an uniform-in-$\e$ estimate of a modified entropy inequality which will significantly be used in later discussion. 
\begin{lemma}\label{L3.2}
Let $T>0$ and suppose that $(f^\varepsilon, n^\varepsilon, v^\varepsilon)$ is a weak solution to the system \eqref{B-1} on the time interval $[0,T)$ in the sense of Definition \ref{D2.1} with the initial data $(f^\e_0, n^\e_0, v^\e_0)$ satisfying \eqref{D2-1.1}. Then we have
\begin{align}
\begin{aligned}\label{A-3}
\mathcal{F}(f^\varepsilon, n^\varepsilon, v^\varepsilon) + \frac{1}{2\varepsilon}& \int_0^t D_1(f^\varepsilon) \,ds  + \int_0^t\int_{\bbr^d} \rho^\varepsilon  |u^\varepsilon - v^\varepsilon|^2 dxds + \int_0^t \int_{\R^d} \nu(n^\e) |\D v^\varepsilon|^2 \,dx ds \\
& \le \mathcal{F}(f^\varepsilon_0, n^\varepsilon_0, v^\varepsilon_0) + C(T)\varepsilon.
\end{aligned}
\end{align}

\end{lemma}

\begin{proof}
Since the proof is almost the same as \cite[Section 5.1]{C-C-K}, we omit its proof here.
\end{proof}

\subsection{Relative entropy estimate}
In this subsection, we provide the relative entropy estimates. For this purpose, we introduce 
\[ 
U = \left( \begin{array}{c} \rho \\ m \\ n \\ w \end{array}\right), \quad A(U) := \left(\begin{array}{cccc} m & 0 & 0 & 0 
\\ (m \otimes m)/\rho & \rho \mathbb{I}_d & 0 & 0  \\
w& 0 & 0 & 0 \\
(w \otimes w)/n & n^\gamma \mathbb{I}_d & 0 & 0\end{array}\right), 
\]
and
\[
F(U) = \left( \begin{array}{c} 0 \\ \rho  (v-u) \\ 0 \\ -\rho  (v-u)  + 2\nabla_x \cdot (\nu(n) \D v) \end{array} \right),
\]
where $\mathbb{I}_d$ denotes the $d \times d$ identity matrix, $m := \rho u$, and $w := nv$, and then we rewrite the system \eqref{B-2} in the form of conservation of laws:
\[
U_t + \nabla_x \cdot A(U) = F(U). 
\]
For notational simplicity, we drop $x$-dependence of differential operators, i.e., $\nabla f := \nabla_x f$ and $\Delta f = \Delta_x f$ for the rest of this paper. The corresponding macroscopic entropy $E(U)$ to above system is given by
\[
E(U) :=  \frac{m^2}{2\rho} + \frac{w^2}{2n} + \rho \log \rho+ H(n), 
\]
and the relative entropy functional $\mathcal{H}$ is defined as
\[ 
\mathcal{H}(V|U) := E(V)-E(U)-DE(U)(V-U), \qquad V = \left( \begin{array}{c} \bar\rho \\ \bar m \\ \bar n \\ \bar w \end{array}\right).
\]
A straightforward computation yields
\[
\mathcal{H}(V|U) = \frac{\bar\rho}{2}|u-\bar{u}|^2 + \frac{\bar{n}}{2} |v-\bar{v}|^2 + P(\bar\rho| \rho) + \tilde{P}(\bar{n}|n), 
\]
where $P(x|y)$ and $\tilde{P}(x|y)$ are relative pressures given by
\[
P(x|y) := x\log x - y\log y + (y-x)(1 +\log y) = \int_{y}^{x} \frac{x - z}{z}\,dz\ge \frac{1}{2} \min \left\{ \frac{1}{x}, \frac{1}{y}\right\} |x-y|^2
\]
and
\[
\tilde{P}(x|y) := \left\{ \begin{array}{ll}
P(x|y) & \textrm{if $\gamma=1$,}\\[2mm]
\displaystyle \frac{1}{\gamma -1} (x^\gamma - y^\gamma) + \frac{\gamma}{\gamma-1}(y-x)y^{\gamma-1} & \textrm{if $\gamma>1$,}
  \end{array} \right.
\]
respectively. Note that 
$$\begin{aligned}
\tilde{P}(x|y) &= K(x)-K(y)-K'(y)(x-y)\cr
&\ge \gamma \min\left\{ x^{\gamma -2}, y^{\gamma-2}\right\} |x-y|^2 \geq \frac{\gamma}{2} \max\{ x^{2-\gamma}, y^{2-\gamma} \}^{-1} |x-y|^2,
\end{aligned}$$
for $\gamma > 1$. Using those newly defined notations, we derive an evolution equation for the relative entropy functional.
\begin{lemma}\label{L3.3}
The relative entropy $\mathcal{H}$ satisfies the following equation:
\begin{align*}
& \int_{\bbr^d} \mathcal{H}(V|U) \,dx + \int_0^t \int_{\bbr^d} \nu(\bar n)|\D(v-\bar v)|^2 \,dxds + \int_0^t \int_{\bbr^d} \bar\rho|( \bar u - \bar v) - (u-v)|^2 \,dxds\\
&\quad =  \int_{\bbr^d}\mathcal{H}(V_0|U_0) \,dx + \int_0^t \int_{\bbr^d} \partial_s E(V) \,dxds + \int_0^t \int_{\R^d} \nu(\bar n) |\D \bar v|^2 \,dx ds + \int_0^t \int_{\bbr^d}\bar\rho | \bar u-  \bar v|^2 \,dxds \\
&\qquad - \int_0^t \int_{\bbr^d} DE(U)(\partial_s V + \nabla \cdot A(V) -F(V)) \,dxds - \int_0^t \int_{\bbr^d} (\nabla DE(U)) : A(V|U) \,dxds\\
& \qquad + \int_0^t\int_{\bbr^d} \left(\frac{ \bar n}{n} \rho - \bar\rho \right)(v- \bar v) (u-v) \,dxds + 2\int_{\bbr^d}\left(\frac{\bar n}{n} - 1\right) \lt(\nabla \cdot (\nu(n) \D v) \rt) \cdot (v - \bar v) \,dx \cr
&\qquad  + 2\int_{\R^d}\lt(\nabla \cdot (\lt(\nu(n) - \nu(\bar n)\rt)\D v) \rt)\cdot (v - \bar v)\,dx,
\end{align*}
where $A:B= \sum_{i=1}^m \sum_{j=1}^n a_{ij} b_{ij}$ for $A = (a_{ij}), B= (b_{ij}) \in \R^{mn}$ and $A(V|U)$ is the relative flux functional defined by
\[
A(V|U) := A(V) - A(U) - DA(U)(V-U). 
\]
\end{lemma}
\begin{proof}
A straightforward computation gives
$$\begin{aligned}
\frac{d}{dt} \int_{\bbr^d} \mathcal{H}(V|U) \,dx & = \int_{\bbr^d} \partial_t E(V) \,dx - \int_{\bbr^d} DE(U)(V_t + \nabla \cdot A(V) -F(V))\, dx\\
& \quad + \int_{\bbr^{d}} D^2 E(U) \nabla \cdot A(U)(V-U) + DE(U)\nabla \cdot A(V) \,dx\\
& \quad - \int_{\bbr^{d}} D^2 E(U) F(U) (V-U) + DE(U)F(V) \,dx\\
& =: \sum_{i=1}^4 I_{i}.
\end{aligned}$$
In order to get the desired result, it suffices to estimate $I_3$ and $I_4$ only. For the estimate of $I_3$, we directly use the idea of \cite[Appendix A]{C-C-K} to get
\[
I_3 = - \int_{\bbr^d} (\nabla DE(U)) : A(V|U)\,dx.
\]
For the estimate of $I_4$, we first notice that
\begin{align*}
D^2 &E(U)F(U)(V-U)\\
& = \left(\begin{array}{cccc}* & -m/\rho^2  & * & 0\\[2mm] 
*  & 1/\rho  & * & 0 \\[2mm]
* & 0 & * & -w/n^2 \\[2mm]
* & 0 & * & 1/n  \end{array} \right) 
\left( \begin{array}{c} 0 \\ \rho (v-u) \\ 0 \\ -\rho (v-u) + 2\nabla \cdot (\nu(n) \D v) \end{array}\right) 
\left(\begin{array}{c} \bar\rho - \rho \\ \bar{m} - m \\ \bar{n} - n \\ \bar{w} - w \end{array}\right)\\
& = -(v-u)\cdot u(\bar\rho - \rho) +  (v-u)\cdot(\bar{m} - m) +\frac{v}{n}\cdot(\rho  (v-u) - 2\nabla \cdot (\nu(n) \D v) ) (\bar{n} - n)\\
& \quad - \frac{1}{n} \lt(\rho (v-u) - 2\nabla \cdot (\nu(n) \D v)\rt) \cdot(\bar{w} - w),
\end{align*}
and
$$\begin{aligned}
DE(U) F(V) &= \bar\rho (\bar{v} - \bar{u}) \cdot u - (\bar\rho  (\bar{v} - \bar{u})  -2\nabla \cdot (\nu(\bar n) \D \bar v)) \cdot v\cr
&=\bar\rho(\bar v - \bar u)\cdot(u-v) + 2\lt(\nabla \cdot (\nu(\bar n) \D \bar v)\rt) \cdot v.
\end{aligned}$$
Combining the above inequalities, we find
\begin{align*}
D^2&E(U)F(U)(V-U) + DE(U)F(V)\\
&=\left( (u-v)\cdot u (\bar\rho - \rho) - (u-v)\cdot(\bar\rho\bar{u} - \rho u) - \frac{\rho v}{n} \cdot (u-v)(\bar{n} - n) + \frac{\rho}{n}(u-v)\cdot((\bar{n})\bar{v} - (n)v) \right)\\
& \quad +  \bar\rho(\bar{v}\cdot u - \bar{u}\cdot u - \bar{v}\cdot v + \bar{u}\cdot v)  -\frac{2}{n} (\bar{n} - n)\lt(\nabla \cdot (\nu(n)\D v)\rt) \cdot v  \cr
&\quad + \frac{2}{n}\lt(\nabla \cdot (\nu(n)\D v)\rt)\cdot (\bar{w} - w) +2\lt( \nabla \cdot (\nu(\bar{n})\D \bar{v})\rt)\cdot v\\
&= \bar{\rho}(u-\bar{u})\cdot(u-v) -\frac{\bar n}{n} \rho(v-\bar{v})\cdot(u-v)\\
& \quad +  \bar\rho(\bar{v}\cdot u - \bar{u}\cdot u - \bar{v}\cdot v + \bar{u} \cdot v)\\
&\quad -\frac2{n} (\bar n - n)(\nabla \cdot (\nu(n)\D v)) \cdot v + \frac2{n} (\nabla \cdot (\nu(n)\D v)) \cdot (\bar w - w) + 2 (\nabla \cdot (\nu(\bar n)\D \bar v))\cdot v \cr
& =: \sum_{i=1}^3 J_{i},
\end{align*}
where $J_2$ can be rewritten as 
$$\begin{aligned}
J_2 &= \bar{\rho}(\bar{v}\cdot u - \bar{u}\cdot u - \bar{v}\cdot v + \bar{u}\cdot v)\\
& = \bar\rho(\bar{v}\cdot u - \bar{u}\cdot u - \bar{v}\cdot v + \bar{u}\cdot v +|u|^2 -u\cdot v + \bar{u}\cdot v -\bar{u}\cdot u) - \bar\rho(|u|^2 -u\cdot v + \bar{u}\cdot v -\bar{u}\cdot u)\\
& = \bar\rho (-2(\bar{u} -\bar{v})\cdot (u-v) -\bar{v}\cdot (u-v) -u\cdot v +|u|^2)- \bar\rho(|u|^2 -u\cdot v + \bar{u}\cdot v -\bar{u}\cdot u)\\
& = \bar\rho (-2(\bar{u} -\bar{v})\cdot (u-v) +|u-v|^2 ) + \bar\rho(-|u-v|^2 -\bar{v}\cdot (u-v) -u\cdot v +|u|^2)\\
& \quad - \bar\rho(|u|^2 -u\cdot v + \bar{u}\cdot v -\bar{u}\cdot u)\\
& = \bar\rho |(\bar{u} - \bar{v}) - (u-v)|^2 - \bar\rho|\bar{u} - \bar{v}|^2 - \bar\rho(u-v)\cdot((u-\bar u) - (v- \bar v)).
\end{aligned}$$
Thus we obtain
\bq\label{est_j}
J_1 + J_2 = \bar\rho|(\bar{u} - \bar{v}) - (u-v)|^2 - \bar\rho|\bar{u} - \bar{v}|^2 + \left( \frac{\bar n}{n}\rho - \bar\rho \right)(v-\bar{v})\cdot(v-u).
\eq
For $J_3$, we estimate
$$\begin{aligned}
J_3 &= -2\lt(\nabla \cdot (\nu(n) \D v)\rt) \cdot (v - \bar v) + 2\lt(\nabla \cdot (\nu(\bar n) \D \bar v)\rt)\cdot v - 2\lt(\frac{\bar n}{n} - 1 \rt)\lt(\nabla \cdot (\nu(n) \D v)\rt) \cdot (v - \bar v)\cr
&=-2\lt(\nabla \cdot (\lt(\nu(n) - \nu(\bar n)\rt)\D v) \rt)\cdot (v - \bar v) -2\lt(\nabla \cdot (\nu(\bar n) \D (v - \bar v)) \rt)\cdot (v - \bar v) \cr
&\quad + 2\lt(\nabla \cdot (\nu(\bar n) \D \bar v)\rt) \cdot \bar v - 2\lt(\frac{\bar n}{n} - 1 \rt)\lt(\nabla \cdot (\nu(n) \D v) \rt)\cdot (v - \bar v),
\end{aligned}$$
which together with \eqref{est_j} gives
\begin{align*}
I_4 &= - \int_{\bbr^d} \bar\rho |(\bar u - \bar v) - (u-v)|^2 \,dx + \int_{\bbr^d} \bar\rho |\bar u - \bar v|^2 dx + \int_{\bbr^d} \left(\frac{\bar n}{n} \rho - \bar\rho \right)(v-\bar v)\cdot (u-v) \,dx\\
& \quad + \int_{\bbr^d} \nu(\bar n)|\D \bar{v}|^2 dx -  \int_{\bbr^d} \nu(\bar n)|\D (v-\bar{v})|^2 dx  \cr
&\quad + 2\int_{\bbr^d}\left(\frac{\bar n}{n} - 1\right) \lt(\nabla \cdot (\nu(n) \D v) \rt) \cdot (v - \bar v) \,dx\cr
&\quad  + 2\int_{\R^d}\lt(\nabla \cdot (\lt(\nu(n) - \nu(\bar n)\rt)\D v) \rt)\cdot (v - \bar v)\,dx.
\end{align*}
This completes the proof.
\end{proof}
%
%
%
\section{Proof of Theorem \ref{T2.1}}\label{sec:3}
In this section, we provide the details of proof of Theorem \ref{T2.1}. Let 
\[
U := \left(\begin{array}{c} \rho \\ \rho u \\ n \\ nv \end{array}\right) \quad \mbox{and} \quad U^\varepsilon := \left(\begin{array}{c} \rho^\varepsilon \\ \rho^\varepsilon u^\varepsilon \\ n^\varepsilon \\ n^\varepsilon v^\varepsilon \end{array}\right), 
\]
where $(f^\varepsilon, n^\varepsilon, v^\varepsilon)$ and $(\rho, u, n,v)$ are weak solutions to the system \eqref{B-1} and a unique strong solution to the system \eqref{B-2}, respectively. Then it follows from Lemma \ref{L3.3} that
\begin{align*}
& \int_{\bbr^d} \mathcal{H}(U^\varepsilon|U)\,dx + \int_0^t \int_{\bbr^d} \nu(n^\e)|\D(v-v^\varepsilon)|^2\,dxds + \int_0^t \int_{\bbr^d} \rho^\varepsilon|( u^\varepsilon - v^\varepsilon) - (u-v)|^2\,dxds\\
&\quad  =  \int_{\bbr^d} \mathcal{H}(U^\varepsilon_0|U_0)\,dx \cr
&\qquad + \int_0^t \int_{\bbr^d} \partial_s E(U^\varepsilon) \,dxds + \int_0^t \int_{\bbr^d} \nu(n^\e)|\D v^\varepsilon|^2 \,dxds + \int_0^t \int_{\bbr^d}\rho^\varepsilon | u^\varepsilon -  v^\varepsilon|^2 \,dxds \\
&\qquad - \int_0^t \int_{\bbr^d} DE(U)(\partial_s U^\varepsilon + \nabla \cdot A(U^\varepsilon) -F(U^\varepsilon)) \,dxds\\
& \qquad - \int_0^t \int_{\bbr^d} (\nabla DE(U)) : A(U^\varepsilon|U) \,dxds\\
& \qquad + \int_0^t\int_{\bbr^d} \left(\frac{ n^\varepsilon}{n} \rho - \rho^\varepsilon \right)(v- v^\varepsilon) \cdot(u-v) \,dxds\\
& \qquad + 2\int_0^t \int_{\bbr^d}\left(\frac{n^\varepsilon - n}{n}\right) \lt(\nabla \cdot (\nu(n) \D v) \rt) \cdot (v - v^\varepsilon) \,dxds\\
& \qquad + 2\int_0^t \int_{\R^d}\lt(\nabla \cdot (\lt(\nu(n) - \nu(n^\e)\rt)\D v) \rt)\cdot (v - v^\e)\,dxds\cr
& =: \sum_{i=1}^7 K_{i}.
\end{align*}
We separately estimate $K_i, i=1,\dots,7$ as follows. \newline

\noindent $\diamond$ (Estimates for $K_1$): It follows from {\bf (H2)} that
\[
K_1 = \mathcal{O}(\sqrt\e).
\]
\noindent $\diamond$ (Estimates for $K_2$): Similar to \cite[Proposition 5.2]{C-C-K}, we estimate 
$$\begin{aligned}
K_2 &=  \int_{\bbr^d} E(U^\varepsilon) \,dx - \mathcal{F}(f^\varepsilon, n^\varepsilon, v^\varepsilon) \\ 
& \quad + \mathcal{F}(f^\varepsilon, n^\varepsilon, v^\varepsilon) + \int_0^t  \int_{\R^d} \nu(n^\e) |\D v^\varepsilon|^2 \,dxds + \int_0^t \int_{\bbr^d}\rho^\varepsilon |u^\varepsilon - v^\varepsilon|^2 \,dxds - \mathcal{F}(f^\e_0, n^\e_0, v^\e_0) \\
&\quad + \mathcal{F}(f^\e_0, n^\e_0, v^\e_0) - \int_{\R^d} E(U_0)\,dx \cr 
& \le C(T)\e + \mathcal{F}(f^\e_0, n^\e_0, v^\e_0)  - \int_{\R^d} E(U_0)\,dx,
\end{aligned}$$
where we used the entropy inequality \eqref{A-3} and the fact that
\[
\int_{\bbr^d} E(U^\varepsilon) \,dx \le \mathcal{F}(f^\varepsilon, n^\varepsilon, v^\varepsilon).
\]
We then use the well-prepared assumption {\bf (H1)} on the initial data to obtain
\[
K_2 \leq C\sqrt\e,
\]
for some $C>0$ independent of $\e$. \newline

\noindent $\diamond$ (Estimates for $K_3$): It follows from \eqref{B-1} that the following holds:
$$\begin{aligned}
&\partial_t \rho^\e + \nabla \cdot (\rho^\e u^\e) = 0,\\
&\partial_t (\rho^\e u^\e) + \nabla \cdot (\rho^\e u^\e \otimes u^\e) + \nabla \rho^\e - \rho^\e(v^\e-u^\e) = \nabla \cdot \lt(\int_{\R^d} (u^\e \otimes u^\e - \xi \otimes \xi + \mathbb{I}_d)f^\e\,d\xi \rt) ,\\
&\partial_t n^\e + \nabla \cdot (n^\e v^\e) = 0,\\
&\partial_t (n^\e v^\e) + \nabla \cdot (n^\e v^\e \otimes v^\e) + \nabla p(n^\e) - 2\nabla \cdot (\nu(n^\e) \D v^\e)  + \rho^\e(v^\e-u^\e) = 0,
\end{aligned}$$
in the sense of distributions. This gives
$$\begin{aligned}
&- \int_0^t \int_{\bbr^d} DE(U)(\partial_s U^\varepsilon + \nabla \cdot A(U^\varepsilon) -F(U^\varepsilon)) \,dxds\cr
&\quad = - \int_0^t \int_{\bbr^d} D_m E(U)\cdot \lt(\nabla \cdot \lt(\int_{\R^d} (u^\e \otimes u^\e - \xi \otimes \xi + \mathbb{I}_d)f^\e\,d\xi \rt) \rt) \,dxds\cr
&\quad = \int_0^t \int_{\bbr^d} \nabla u : \lt(\int_{\R^d} (u^\e \otimes u^\e - \xi \otimes \xi + \mathbb{I}_d)f^\e\,d\xi \rt)  \,dxds
\end{aligned}$$
due to $D_m E(U) = u$. We then follow the proof of \cite[Lemma 4.4]{K-M-T2} to get
\[
K_3 \le C\sqrt{\varepsilon},
\]
where $C =  C(\|\nabla u\|_{L^\infty})$ is a positive constant independent of $\e$.  \newline

\noindent $\diamond$ (Estimates for $K_4$): Note that
\begin{align*}
A(U^\varepsilon|U) &= A(U^\varepsilon)-A(U) - DA(U)(U^\varepsilon-U)\\
&= \left(\begin{array}{cccc}0 & 0 &0 &0 \\ \rho^\varepsilon( u^\varepsilon - u) \otimes (u^\varepsilon - u) & 0 & 0 & 0 \\ 0 & 0 &0 &0 \\
n^\varepsilon(v^\varepsilon - v)\otimes (v^\varepsilon - v) & (\gamma-1)\tilde{P}(n^\varepsilon | n) \mathbb{I}_d & 0 & 0  \end{array}\right).
\end{align*}
This implies
$$\begin{aligned}
\int_{\bbr^d} |A(U^\varepsilon|U)| \,dx &\le  \int_{\bbr^{d}} \rho^\varepsilon|u^\varepsilon - u|^2 + n^\varepsilon |v^\varepsilon-v|^2 + d(\gamma-1)\tilde{P}(n^\varepsilon | n) \,dx\cr
& \le C \int_{\bbr^d} \mathcal{H}(U^\varepsilon|U) \,dx,
\end{aligned}$$
where $C>0$ only depends on $d$ and $\gamma$. Thus we obtain
\[
K_4 \le C \int_0^t \int_{\bbr^d} \mathcal{H}(U^\varepsilon|U) \,dx ds.
\]
$\diamond$ (Estimates for $K_5$): We divide $K_5$ into two terms:
\begin{align*}
K_5 &= \int_0^t \int_{\bbr^d} (\rho-\rho^\varepsilon)(v- v^\varepsilon)\cdot(u-v) \,dxds + \int_0^t \int_{\bbr^d} \rho \left( \frac{ n^\varepsilon - n}{n}\right)(v- v^\varepsilon)\cdot (u-v) \,dx ds\\
& =: K_5^1 + K_5^2.
\end{align*}
For the estimate of $K_5^1$, we use the following elementary inequality
\bq\label{minmax}
1 = \min\left\{x^{-1}, y^{-1} \right\}\max\left\{x, y \right\} \le \min\left\{x^{-1}, y^{-1} \right\} (x+y) \quad \mbox{for} \quad x,y>0,
\eq
to get
\begin{align*}
& \lt|\int_{\bbr^d} (\rho-\rho^\varepsilon)(v- v^\varepsilon)\cdot(u-v) \,dx \rt| \\
& \quad \le \left( \int_{\bbr^d} \min\left\{\frac{1}{\rho^\varepsilon}, \frac{1}{\rho} \right\} (\rho- \rho^\varepsilon)^2\,dx \right)^{1/2} \left(\int_{\bbr^d} (\rho + \rho^\varepsilon)|v- v^\varepsilon|^2 |u-v|^2 \,dx\right)^{1/2}\\
& \quad \le C \left( \int_{\bbr^d} \mathcal{H}(U^\varepsilon|U) \,dx \right)^{1/2}\left(\int_{\bbr^d} (\rho + \rho^\varepsilon)|v- v^\varepsilon|^2 |u-v|^2 \,dx\right)^{1/2}.
\end{align*}
On the other hand, the second term on the above inequality can be estimated as
\begin{align*}
\int_{\bbr^d}& (\rho + \rho^\varepsilon)|v-v^\varepsilon|^2 |u-v|^2 \,dx\\
& \le \|\rho\|_{L^\infty}\|v-v^\varepsilon\|_{L^{p*}}^2 \|u-v\|_{L^d}^2 + 2\int_{\bbr^d} \Big(\rho^\varepsilon|(u-u^\varepsilon) - (v-v^\varepsilon)|^2 + \rho^\varepsilon|u-u^\varepsilon|^2 \Big)|u-v|^2 \,dx\\
&  \le C \|\nabla(v - v^\varepsilon)\|_{L^2}^2 \|u-v\|_{L^d}^2 + 2\|u-v\|_{L^\infty}^2 \left(\int_{\bbr^d} \rho^\varepsilon|(u-u^\varepsilon) - (v-v^\varepsilon)|^2 \,dx + \int_{\bbr^d} \rho^\varepsilon|u-u^\varepsilon|^2 \,dx\right),
\end{align*}
where $1/p^* = 1/2 - 1/d$ and we used Gagliardo-Nirenberg-Sobolev inequality. Note that
\[
\|\rho\|_{L^\infty} \leq C\|\rho\|_{H^s},\quad \|u-v\|_{L^d} \leq \|u-v\|_{L^\infty}^{(d-2)/d}\|u-v\|_{L^2}^{2/d} \leq C\|u-v\|_{H^s},
\]
due to $s > d/2 + 2$, and
$$\begin{aligned}
\frac12\int_{\R^d} |\nabla (v - v^\e)|^2\,dx &\leq \frac12\int_{\R^d} |\nabla (v - v^\e)|^2\,dx + \frac12\int_{\R^d} |\nabla \cdot (v - v^\e)|^2\,dx \cr
&\leq \int_{\R^d} |\D (v - v^\e)|^2\,dx\cr
&\leq \frac{1}{\nu_*}\int_{\R^d} \nu(n^\e)|\D (v - v^\e)|^2\,dx.
\end{aligned}$$
These together with using Young's inequality give
$$\begin{aligned}
K_5^1 &\le C\int_0^t \int_{\bbr^d} \mathcal{H}(U^\varepsilon|U)\,dxds + \frac18\int_0^t \int_{\bbr^d}\nu(n^\e)|\D (v - v^\e)|^2\,dxds  \cr
&\quad + \frac{1}{2}\int_0^t \int_{\bbr^d} \rho^\varepsilon|(u-u^\varepsilon) - (v-v^\varepsilon)|^2 \,dxds,
\end{aligned}$$
where $C = C(\|\rho\|_{L^\infty}, \|u-v\|_{L^\infty(0,T;L^d \cap L^\infty)}, \nu_*)$ is a positive constant. For the term $K_5^2$, we let $n_* := \inf_{x \in \bbr^d} n(x) > 0$ and use the inequality \eqref{minmax} to get
\begin{align*}
&\lt|\int_{\bbr^d} \rho \left( \frac{ n^\varepsilon - n}{n}\right)(v- v^\varepsilon)\cdot(u-v) \,dx\rt|\\
&\quad\le \frac{\|\rho\|_{L^\infty}}{n_*}\int_{\bbr^d} |n^\varepsilon-n| |v^\varepsilon-v| |u-v| \,dx \\
&\quad\le C \left( \int_{\bbr^d} \min\left\{(n^\varepsilon)^{\gamma-2}, n^{\gamma-2} \right\} (n- n^\varepsilon)^2\,dx \right)^{1/2}\cr
&\qquad \qquad \times  \left(\int_{\bbr^d} \lt(n^{2-\gamma} + (n^\varepsilon)^{2-\gamma}\rt)|v- v^\varepsilon|^2 |u-v|^2\,dx\right)^{1/2}\\
&\quad \le C \left( \int_{\bbr^d} \mathcal{H}(U^\varepsilon|U) \,dx \right)^{1/2} \left(\int_{\bbr^d} \lt(n^{2-\gamma} + (n^\varepsilon)^{2-\gamma}\rt)|v- v^\varepsilon|^2  |u-v|^2 \,dx\right)^{1/2},
\end{align*}
where $C = C(\|\rho\|_{L^\infty}, n_*, \gamma)$ is a positive constant. We further estimate
\begin{align*}
\int_{\bbr^d}& \lt(n^{2-\gamma} + (n^\varepsilon)^{2-\gamma}\rt)|v- v^\varepsilon|^2 |u-v|^2 \,dx \\
&\le \|n\|_{L^\infty}^{2-\gamma} \|v-v^\varepsilon\|_{L^{p^*}}^2 \|u-v\|_{L^d}^2 + \int_{\bbr^d} (n^{\varepsilon})^{2-\gamma} |v-v^\varepsilon|^2 |u-v|^2 \,dx\\
&\le C\|n\|_{L^\infty}^{2-\gamma}\|u-v\|_{L^d}^2\|\nabla(v-v^\varepsilon)\|_{L^2}^2 + \int_{\bbr^d} (n^{\varepsilon})^{2-\gamma} |v-v^\varepsilon|^2 |u-v|^2 \,dx.
\end{align*}

\noindent For $\gamma=1$ or $2$, we easily get 
\begin{displaymath}
\int_{\R^d} (n^\varepsilon)^{2-\gamma}|v- v^\varepsilon|^2 |u-v|^2 \,dx \leq \left\{ \begin{array}{ll}
\displaystyle \|u-v\|_{L^\infty}^2\int_{\bbr^d} \mathcal{H}(U^\varepsilon|U) \,dx & \textrm{for $\gamma=1$,}\\[4mm]
\|u-v\|_{L^d}^2\|\nabla(v-v^\varepsilon)\|_{L^2}^2 & \textrm{for $\gamma=2$}.
  \end{array} \right.
\end{displaymath}
For $\gamma \in (1,2)$, we first use Young's inequality to obtain
\begin{align*}
\int_{\bbr^d}& (n^\varepsilon)^{2-\gamma}|v- v^\varepsilon|^2 |u-v|^2 \,dx \\
&\le \int_{\bbr^d} (n^{\varepsilon})^{2-\gamma} |v-v^\varepsilon|^{(4-2\gamma) + (2\gamma -2)} |u-v|^2 \,dx\\
& \le(2-\gamma)\int_{\bbr^d} n^\varepsilon |v-v^\varepsilon|^2 \,dx + (\gamma-1) \int_{\bbr^d}|v-v^\varepsilon|^2 |u-v|^{2/(\gamma-1)} \,dx \\
& \le(2-\gamma)\int_{\bbr^d} n^\varepsilon|v-v^\varepsilon|^2 \,dx + (\gamma-1)\|v-v^\varepsilon\|_{L^{p^*}}^2 \|u-v\|_{L^{\frac{d}{\gamma-1}}}^{\frac{2}{\gamma-1}}\\
& \le C \int_{\bbr^d} \mathcal{H}(U^\varepsilon|U) \,dx + C\int_{\R^d} |\nabla (v - v^\e)|^2\,dx,
\end{align*}
where $C = C(\gamma, \|u-v\|_{L^\infty(0,T;L^d \cap L^\infty)})$ is a positive constant. Note that $d/(\gamma-1) > d > 2$. Using the similar argument as in the estimate of $K_5^1$, we find   
\[ 
K_5^2 \le C \int_0^t \int_{\bbr^d} \mathcal{H}(U^\varepsilon|U)\,dxds + \frac18\int_0^t \int_{\bbr^d}\nu(n^\e)|\D (v - v^\e)|^2\,dxds,
\]
for any $\gamma \in [1,2]$. Thus, we collect the estimates for $K_5^1$ and $K_5^2$ to yield
$$\begin{aligned}
K_5 &\le C\int_0^t \int_{\bbr^d} \mathcal{H}(U^\varepsilon|U) \,dxds + \frac14\int_0^t \int_{\bbr^d}\nu(n^\e)|\D (v - v^\e)|^2\,dxds \cr
&\quad + \frac12\int_0^t \int_{\bbr^d} \rho^\varepsilon|(u-u^\varepsilon) - (v-v^\varepsilon)|^2 \,dxds,\\
\end{aligned}$$
where $C = C(\gamma, n_*, \nu_*, \|\rho\|_{L^\infty}, \|u-v\|_{L^\infty(0,T;L^d \cap L^\infty)})$ is a positive constant. \newline

\noindent $\diamond$ (Estimates for $K_6$): By using almost the same argument as in the estimate of $K_5^2$, we have 
\begin{align*}
&2\lt|\int_{\bbr^d}\left(\frac{n^\varepsilon - n}{n}\right) \lt(\nabla \cdot (\nu(n) \D v)\rt) \cdot (v - v^\varepsilon) \,dx\rt|\\
&\quad \le \frac{2}{n_*} \int_{\bbr^d} |n^\varepsilon-n| |v-v^\varepsilon| |\nabla \cdot (\nu(n) \D v)| \, dx\\
&\quad \le  C \left( \int_{\bbr^d} \min\left\{(n^\varepsilon)^{\gamma-2}, n^{\gamma-2} \right\} (n- n^\varepsilon)^2\,dx \right)^{1/2} \cr
&\qquad \qquad \times \left(\int_{\bbr^d} (n^{2-\gamma} + (n^\varepsilon)^{2-\gamma})|v- v^\varepsilon|^2 |\nabla \cdot (\nu(n) \D v)|^2\,dx\right)^{1/2}\\
&\quad \le C \int_{\bbr^d} \mathcal{H}(U^\varepsilon|U) \,dx + \frac18\int_{\bbr^d}\nu(n^\e)|\D (v - v^\e)|^2\,dx.
\end{align*}
This yields
\[
K_6 \le C \int_0^t \int_{\bbr^d} \mathcal{H}(U^\varepsilon|U) \,dxds + \frac18\int_0^t \int_{\bbr^d}\nu(n^\e)|\D (v - v^\e)|^2\,dxds, 
\]
where $C = C(\gamma, n_*, \nu_*, \|\rho\|_{L^\infty}, \|\nabla \cdot (\nu(n) \D v)\|_{L^\infty(0,T;L^d \cap L^\infty)})$ is a positive constant. We notice that
$$\begin{aligned}
\|\nabla \cdot (\nu(n) \D v)\|_{L^d \cap L^\infty} &\leq C\|\nabla n\|_{L^\infty}\|\D v\|_{L^d \cap L^\infty} + C\|\nabla \D v\|_{L^d \cap L^\infty} \cr
&\leq C\|\nabla n\|_{L^\infty}\|\D v\|_{H^{s-2}} + C\|\nabla \D v\|_{H^{s-2}} \cr
&\leq C(1 + \|\nabla n\|_{H^{s-1}})\|\nabla v\|_{H^{s-1}},
\end{aligned}$$
due to $\nu \in \mc^1(\R_+)$, $d > 2$, and $s-2 > d/2$. \newline

\noindent $\diamond$ (Estimates for $K_7$): Using the integration by parts and symmetry of $\D v$, we find
\[
2\int_{\R^d}\lt(\nabla \cdot (\lt(\nu(n) - \nu(n^\e)\rt)\D v) \rt)\cdot (v - v^\e)\,dx = - \int_{\R^d} \lt(\nu(n) - \nu(n^\e)\rt)\D v: \D (v - v^\e)\,dx.
\]
Then we estimate
$$\begin{aligned}
&\lt|\int_{\R^d} \lt(\nu(n) - \nu(n^\e)\rt)\D v: \D (v - v^\e)\,dx\rt|\cr
&\quad \leq \nu_{\mbox{\tiny Lip}}\|\D v\|_{L^\infty}\int_{\R^d} |\D(v - v^\e)| |n - n^\e|\,dx\cr
&\quad \leq \nu_{\mbox{\tiny Lip}}\|\D v\|_{L^\infty}\lt(\int_{\bbr^d} \min\left\{(n^\varepsilon)^{\gamma-2}, n^{\gamma-2} \right\} (n- n^\varepsilon)^2\,dx \rt)^{1/2}\cr
&\qquad \qquad \times \lt(\int_{\bbr^d} \lt(n^{2-\gamma} + (n^\varepsilon)^{2-\gamma}\rt)|\D(v - v^\e)|^2\,dx \rt)^{1/2}\cr
&\quad \leq C\nu_{\mbox{\tiny Lip}}\|\D v\|_{L^\infty}\lt(\int_{\R^d} \mathcal{H}(U^\varepsilon|U) \,dx \rt)^{1/2}\lt(\int_{\bbr^d} (n^{2-\gamma} + (n^\varepsilon)^{2-\gamma})|\D(v - v^\e)|^2\,dx \rt)^{1/2}.
\end{aligned}$$
On the other hand, by using the assumption on $\nu$ \eqref{ass_nu}, we get $(n^\varepsilon)^{2-\gamma} \leq c_0\nu(n^\e)$, and this gives
$$\begin{aligned}
&\int_{\bbr^d} n^{2-\gamma} + (n^\varepsilon)^{2-\gamma})|\D(v - v^\e)|^2\,dx \cr
&\quad \leq \frac{\|n\|_{L^\infty}^{2-\gamma}}{\nu_*}\int_{\R^d} \nu(n^\e)|\D(v-v^\e)|^2\,dx + c_0 \int_{\R^d} \nu(n^\e)|\D(v - v^\varepsilon)|^2 \,dx.
\end{aligned}$$
This together with using Young's inequality provides
\[
K_7 \leq \frac{1}{8}\int_{\R^d} \nu(n^\e)|\D (v-v^\e)|^2\,dx + C\int_{\R^d} \mathcal{H}(U^\varepsilon|U) \,dx, 
\]
where $C = C(\nu_{\mbox{\tiny Lip}}, \nu_*, c_0, \|\D v\|_{L^\infty}, \|n\|_{L^\infty}, \gamma)$ is a positive constant independent of $\e$. 

By combining all of the above estimates, we have
\begin{align*}
&\int_{\bbr^d} \mathcal{H}(U^\varepsilon|U)\,dx + \frac12 \int_0^t \int_{\R^d} \nu(n^\e)|\D (v-v^\e)|^2\,dxds +\frac{1}{2} \int_0^t \int_{\bbr^d} \rho^\varepsilon|( u^\varepsilon - v^\varepsilon) - (u-v)|^2\,dxds\\
&\quad \le  C\lt(\int_0^t \int_{\bbr^d} \mathcal{H}(U^\varepsilon|U)\,dxds + \sqrt{\varepsilon} \rt),
\end{align*}
where $C = C(\nu_{\mbox{\tiny Lip}}, c_0, \gamma, n_*, \nu_*, \|\rho\|_{L^\infty}, \|u-v\|_{L^d \cap L^\infty}, \|n\|_{L^\infty}, \|\D v\|_{L^\infty}, \|\nabla \cdot (\nu(n) \D v)\|_{L^d \cap L^\infty}, \|\nabla u\|_{L^\infty})$ is a positive constant.
Finally, we apply Gr\"onwall's inequality to the above to conclude the desired result.\newline

We next provide the strong convergence appeared in Theorem \ref{T2.1} by using the relative entropy inequality \eqref{res_conv}. Since the convergence of $\rho^\e$, $\rho^\e u^\e$, and $\rho^\e|u^\e|^2$ can be obtained by the same argument as in \cite{K-M-T2}, we only show the strong convergence of $n^\e$, $n^\e v^\e$ and $n^\e |v^\e|^2$ below. \newline

\noindent $\diamond$ (Convergence of $n^\e$ to $n$): Before proceding, we claim that the following inequality holds: if $x, y>0$ and $0 < y_{min} \le y \le y_{max}<\infty$, then
\begin{equation}\label{T1-1.1}
\tilde{P}(x|y) = K(x)-K(y)-K'(y)(x-y) \ge  \left\{ \begin{array}{ll}
\gamma (2y_{max})^{\gamma-2}|x-y|^2 & \textrm{if $y/2 \le x \le 2y$,}\\[2mm]
\displaystyle \frac{\gamma y_{min}^\gamma}{4(1+y_{min}^\gamma)}(1+x^\gamma) & \textrm{otherwise.}
  \end{array} \right.
\end{equation}
\noindent If $y/2 \le x \le 2y$, we easily find
\begin{align*}
K(x)-K(y)-K'(y)(x-y) &\ge \gamma \min\{x^{\gamma-2},y^{\gamma-2}\} |x-y|^2\\
& \ge \gamma (2y)^{\gamma-2} |x-y|^2 \ge \gamma (2y_{max})^{\gamma-2} |x-y|^2.
\end{align*}
If $x >2y > y \ (> y_{min})$, i.e., $y/x < 1/2$, we get
\begin{align*}
K(x)-K(y)-K'(y)(x-y) &\ge \gamma \min\{x^{\gamma-2},y^{\gamma-2}\} |x-y|^2\\
& = \gamma x^{\gamma-2} |x-y|^2 = \gamma x^\gamma \lt| 1-\frac{y}{x}\rt|^2\\
& \ge \frac{\gamma x^\gamma}{4} = \frac{\gamma}{4}(1+x^\gamma) \lt(1 - \frac{1}{1+x^\gamma} \rt)\\
& \ge \frac{\gamma}{4}(1+x^\gamma)\lt(1 - \frac{1}{1+y_{min}^\gamma} \rt).
\end{align*}
On the other hand, if $x< y/2$, i.e., $x/y < 1/2$, we obtain
\begin{align*}
K(x)-K(y)-K'(y)(x-y) &\ge \gamma y^{\gamma-2} |x-y|^2 = \gamma y^\gamma \lt| 1-\frac{x}{y}\rt|^2\\
& \ge \frac{\gamma y^\gamma}{4} = \frac{\gamma}{4}(1+y^\gamma) \lt(1 - \frac{1}{1+y^\gamma} \rt)\\
& \ge \frac{\gamma}{4}(1+x^\gamma)\lt(1 - \frac{1}{1+y_{min}^\gamma} \rt).
\end{align*}
Thus we have the inequality \eqref{T1-1.1}. We now use that inequality \eqref{T1-1.1} to show the convergence of $n^\e$ to $n$. For $\Omega \subset \bbr^d$ with $|\Omega| < \infty$, we estimate
\begin{align*}
\int_\Omega |n^\e - n|^\gamma \,dx &= \int_{\Omega \cap \{n/2 \le n^\e \le 2n \}} |n^\e -n|^\gamma \,dx + \int_{\Omega \cap\{n/2 \le n^\e \le 2n \}^c} |n^\e - n|^\gamma \,dx\cr
& =: L^\e_1 + L^\e_2. 
\end{align*}
For $L^\e_1$, we find
\begin{align*}
L^\e_1 &\le \left(\int_{\Omega \cap \{n/2 \le n^\e \le 2n \}} \min\{(n^\e)^{\gamma-2}, n^{\gamma-2}\} |n^\e - n|^2 \,dx\right)^{\frac{\gamma}{2}} \left(\int_{\Omega \cap \{n/2 \le n^\e \le 2n \}} \max\{(n^\e)^\gamma, n^\gamma\} \,dx\right)^{\frac{2-\gamma}{2}}\\
&\le C\left(\int_{\Omega \cap \{n/2 \le n^\e \le 2n \}} \mathcal{H}(U^\e|U) \,dx\right)^{\frac{\gamma}{2}} \lt( \Big(2\|n\|_{L^\infty}\Big)^\gamma|\Omega|\rt)^{\frac{2-\gamma}{2}} \longrightarrow 0,
\end{align*}
as $\e \to 0$, where $C = C(\gamma)$ is a positive constant independent of $\e$. For $L^\e_2$, we use \eqref{T1-1.1} to get
\begin{align*}
L^\e_2 &\le \int_{\Omega \cap\{n/2 \le n^\e \le 2n \}^c} \|n\|_{L^\infty}^\gamma \lt|\frac{n^\e}{n} + 1\rt|^\gamma dx\\
&\le \int_{\Omega \cap\{n/2 \le n^\e \le 2n \}^c} (2\|n\|_{L^\infty})^\gamma \left(\left(\frac{n^\e}{n}\right)^\gamma + 1 \right) dx\\
&\le \int_{\Omega \cap\{n/2 \le n^\e \le 2n \}^c} (2\|n\|_{L^\infty})^\gamma \left(\left(\frac{n^\e}{n_*}\right)^\gamma + 1 \right) dx\\
&\le C\int_{\Omega \cap\{n/2 \le n^\e \le 2n \}^c} (1+(n^\e)^\gamma) \,dx\\
&\le C \int_{\Omega \cap\{n/2 \le n^\e \le 2n \}^c} H(U^\e | U)\,dx \longrightarrow 0,
\end{align*}
as $\e \to 0$, where $C = C(\|n\|_{L^\infty}, n_*, \gamma)$ is a positive constant independent of $\e$. Thus we have the convergence $n^\e \to n$ in $L_{loc}^1(0,T;L_{loc}^\gamma(\bbr^d))$, and this together with the integrability condition yields that it also holds in $L_{loc}^1(0,T;L_{loc}^p(\bbr^d))$ with $p \in [1,\gamma]$.\newline

\noindent $\diamond$ (Convergence of $n^\e v^\e$ to $nv$): For $\Omega \subseteq \bbr^d$ with $|\Omega| < \infty$, similarly as before, we estimate
\begin{align*}
\int_\Omega |n^\e v^\e - n v| \,dx & \le \int_\Omega (n^\e| v^\e - v| + |n^\e - n| |v|) \,dx\\
&=:L^\e_3 + L^\e_4,
\end{align*}
where $L^\e_3$ can be bounded by 
\[
L^\e_3 \le \left( \int_\Omega n^\e |v^\e - v|^2 \,dx\right)^{1/2} \left( \int_\Omega n^\e \,dx \right)^{1/2}.
\]
Note that $n^\e$ is locally integrable in $\R^d$, and furthermore, we find
\begin{align*}
 \int_\Omega n^\e \,dx  &= \int_{\Omega \cap \{n/2 \le n^\e \le 2n \}} n^\e \,dx +  \int_{\Omega \cap\{n/2 \le n^\e \le 2n \}^c} n^\e \,dx\\
 & \le 2\|n\|_{L^\infty}|\Omega| + |\Omega|^{\frac{\gamma-1}{\gamma}} \left( \int_{\Omega \cap\{n/2 \le n^\e \le 2n \}^c} (n^\e)^\gamma \,dx\right)^{\frac{1}{\gamma}}\\
 & \le 2\|n\|_{L^\infty}|\Omega| + |\Omega|^{\frac{\gamma-1}{\gamma}} \left( \int_{\Omega \cap\{n/2 \le n^\e \le 2n \}^c} \mathcal{H}(U^\e|U) \,dx\right)^{\frac{1}{\gamma}}.
\end{align*}
This gives $L^\e_3 \to 0$ as $\e \to 0$. For the estimate of $L^\e_4$, we obtain
\[
L^\e_4 \le \|v\|_{L^\infty}|\Omega|^{\frac{\gamma-1}{\gamma}} \lt(\int_\Omega |n^\e - n|^\gamma \,dx\rt)^{1/\gamma} \longrightarrow 0,
\]
as $\e \to 0$. This gives the desired result for the convergence of $n^\e v^\e$.\newline

\noindent $\diamond$ (Convergence of $n^\e|v^\e|^2$ to $n|v|^2$): Note that the following identity holds:
\[
n^\e |v^\e|^2 - n|v|^2 = n^\e |v^\e - v|^2 + 2v \cdot (n^\e v^\e - nv) + |v|^2(n-n^\e). 
\]
This relation together with the previous convergence results yields the desired strong convergence of $n^\e|v^\e|^2$. This completes the proof.

%
%
%
%

\section*{Acknowledgments}
The work of Y.-P. Choi is supported by National Research Foundation of Korea(NRF) grant funded by the Korea government(MSIP) (No. 2017R1C1B2012918 and 2017R1A4A1014735) and POSCO Science Fellowship of POSCO TJ Park Foundation. The work of J. Jung is supported by the German Research Foundation (DFG) under the project number IRTG2235.

%
%
%
%


\begin{thebibliography}{10}


\bibitem{B-C-H-K1} H.-O. Bae, Y.-P. Choi, S.-Y. Ha, and M.-J. Kang: \textit{ Global existence of strong solution for the Cucker-Smale-Navier-Stokes system}, J. Differ. Equ. {\bf 257} (2014), 2225-2255.

\bibitem{B-C-H-K2} H.-O. Bae, Y.-P. Choi, S.-Y. Ha, and M.-J. Kang, \textit{Asymptotic flocking dynamics of Cucker-Smale particles immersed in compressible fluids}, Discrete Contin. Dyn. Syst., Ser. A {\bf 34} (2014), 4419-4458.

\bibitem{B-D-G-M} L. Boudin, L. Desvillettes, C. Grandmont, and A. Moussa, \textit{Global existence of solution for the coupled Vlasov and Naiver-Stokes equations}, Differ. Integral Equ. {\bf 22} (2009), 1247-1271.

\bibitem{BNV16} D. Bresch, P. Noble, and J.-P. Vila, \textit{Relative entropy for compressible Navier-Stokes equations with density-dependent viscosities and applications}, C. R. Acad. Sci. Paris, Ser. I {\bf 354} (2016), 45-49.

\bibitem{BNV17} D. Bresch, P. Noble, and J.-P. Vila, \textit{Relative entropy for compressible Navier-Stokes equations with density dependent viscosities and various applications}, ESAIM: Proc. {\bf 58} (2017), 40-57.

\bibitem{C-C-K} J. A. Carrillo, Y.-P. Choi, and T. K. Karper: \textit{On the analysis of a coupled kinetic-fluid model with local alignment forces}, Ann. I. H. Poincar\'e - AN. {\bf 33}, (2016), 273-307.


\bibitem{C-D-M} J.A. Carrillo, R. Duan, and A. Moussa: \textit{Global classical solutions close to the equilibrium to the Vlasov-Fokker-Planck-Euler system}, Kinet. Relat. Models {\bf 4} (2011), 227-258.


\bibitem{C-G} J.A. Carrillo and T. Goudon: \textit{Stability and asymptotic analysis of a fluid-particle interaction model}, Commun. Partial Differ. Equ. {\bf 31} (2006), 1349-1379.

\bibitem{C0} Y.-P. Choi: \textit{Compressible Euler equations interacting with incompressible flow}, Kinet. Relat. Models {\bf 8} (2015), 335-358.

\bibitem{C} Y.-P. Choi: \textit{Global classical solutions and large-time behavior of the two-phase fluid model}, SIAM J. Math. Anal. {\bf 48} (2016), 3090-3122.

\bibitem{C2} Y.-P. Choi: \textit{Global classical solutions of the Vlasov-Fokker-Planck equation with local alignment forces}, Nonlinearity, {\bf 29} (2016), 1887-1916.

\bibitem{C3} Y.-P. Choi: \textit{Large-time behavior for the Vlasov/compressible Navier-Stokes equations}, J. Math. Phys. {\bf 57}, 071501 (2016).

\bibitem{C4} Y.-P. Choi: \textit{Finite-time blow-up phenomena of Vlasov/Navier-Stokes equations and related systems}, J. Math. Pures Appl. {\bf 108} (2017), 991-1021.

\bibitem{CHJKpre} Y.-P. Choi, S.-Y. Ha, J. Jung, and J. Kim: \textit{Global dynamics of the thermodynamic Cucker-Smale ensemble immersed in incompressible viscous fluid}, Nonlinearity, to appear.

\bibitem{CHJKpre2} Y.-P. Choi, S.-Y. Ha, J. Jung, and J. Kim: \textit{On the coupling of kinetic thermomechanical Cucker-Smale equation and compressible viscous fluid system}, preprint.

\bibitem{CYpre} Y.-P. Choi and S.-B. Yun: \textit{Global existence of weak solutions for Navier-Stokes-BGK system}, preprint.

\bibitem{Da} C. M. Dafermos: \textit{The second law of thermodynamics and stability}, Arch. Ration. Mech. Anal. {\bf 70} (1979), 167-179.

\bibitem{FJN12} E. Feireisl, B. J. Jin, and A. Novotn\'y: \textit{Relative entropies, suitable weak solutions, and weak-strong uniqueness for the compressible Navier-Stokes system}, J. Math. Fluid Mech. {\bf 14} (2012), 717-730.

\bibitem{G-L-T} P. Goncalves, C. Landim, and C. Toninelli: \textit{Hydrodynamic limit for a particle system with degenerate rates}, Ann. Inst. Henri Poincar\'e Probab. Stat. {\bf 45} (2009), 887-909.

\bibitem{G-H-M-Z} T. Goudon, L. He, A. Moussa, and P. Zhang: \textit{The Navier-Stokes-Vlasov-Fokker-Planck system near equilibrium}, SIAM J. Math. Anal. {\bf 42} (2010), 2177-2202.

\bibitem{G-J-V} T. Goudon, P.-E. Jabin, and A. Vasseur: \textit{Hydrodynamic limit for the Vlasov-Navier-Stokes equations: I. Light particles regime}, Indiana Univ. Math. J. {\bf 53} (2004), 1495-1515.


\bibitem{G-J-V2} T. Goudon, P.-E. Jabin, and A. Vasseur: \textit{Hydrodynamic limit for the Vlasov-Navier-Stokes equations: II. Fine particles regime}, Indiana Univ. Math. J. {\bf 53} (2004), 1517-1536.





\bibitem{K-M-T2} T.K. Karper, A. Mellet and K. Trivisa: \textit{Hydrodynamic limit of the kinetic Cucker-Smale flocking model}, Math. Models Methods Appl. Sci. {\bf 25} (2014), 131-163.


\bibitem{L} C. Landim: \textit{Hydrodynamic limit of interacting particle systems}, in: School and Conference on Probability Theory, in: ICTP Lect. Notes, vol. XVII, Abdus Salam Int. Cent. Theoret. Phys., Trieste, 2004, 57100 (electronic).

\bibitem{M84} A. Majda: \textit{Compressible fluid flow and systems of conservation laws in several space variables}, Applied Mathematical Sciences, vol. 53, Springer-Verlag, New York, 1984.

\bibitem{Math10} J. Mathiaud: \textit{Local smooth solutions of a thin spray model with collisions}, Math. Mod. Meth. Appl. Sci. {\bf 20} (2010), 191-221.

\bibitem{M-V} A. Mellet and A. Vasseur: \textit{Global weak solutions for a Vlasov-Fokker-Planck/Navier-Stokes system of equations}, Math. Models Methods Appl. Sci. {\bf 17} (2007), 1039-1063.

\bibitem{M-V2} A. Mellet and A. Vasseur: \textit{Asymptotic analysis for a Vlasov-Fokker-Planck/compressible Navier-Stokes equations}, Commun. Math. Phys. {\bf 281} (2008), 573-596.

\bibitem{WY14} D. Wang and C. Yu: \textit{Global weak solutions to the inhomogeneous Navier-Stokes-Vlasov equations}, J. Differ. Equ. {\bf 259} (2014), 3976-4008.


\bibitem{Y} H. T. Yau: \textit{Relative entropy and hydrodynamics of Ginzburg-Landau models}, Lett. Math. Phys. {\bf 22} (1991), 6380.

\bibitem{Yu13} C. Yu: \textit{Global weak solutions to the incompressible Navier-Stokes-Vlasov equations}, J. Math. Pures Appl. {\bf 100} (2013), 275-293.




\end{thebibliography}
\end{document}